\documentclass[reqno]{amsart}

\usepackage{amsmath}
\usepackage{amsfonts}
\usepackage{amsthm}
\usepackage{amssymb}
\usepackage{graphicx}
\usepackage{graphics}
\usepackage{bm}
\usepackage{dsfont}
\usepackage{color}
\usepackage{enumerate}
\usepackage[font=footnotesize]{caption}
\usepackage[all]{xy}
\usepackage{xypic}
\usepackage{tikz}
\usepackage{pgfplots}
\usepackage{paralist}
\usepackage{hyperref}
\allowdisplaybreaks

\numberwithin{equation}{section}

\newtheorem{thm}{Theorem}[section]
\newtheorem*{thm*}{Theorem}
\newtheorem{cor}[thm]{Corollary}
\newtheorem{lem}[thm]{Lemma}
\newtheorem{prop}[thm]{Proposition}
\theoremstyle{definition}
\newtheorem{exm}[thm]{Example}
\newtheorem{qst}[thm]{Question}

\theoremstyle{remark}
\newtheorem{rem}[thm]{Remark}
\newcommand{\N}{\mathds{N}}

\newcommand{\Z}{\mathds{Z}}

\newcommand{\R}{\mathds{R}}

\newcommand{\OO}{\mathcal{O}}
\newcommand{\C}{\mathds{C}}

\newcommand{\Ss}{S}

\newcommand{\HG}{H_{\mathrm{bas}\:G}}

\newcommand{\To}{\rightarrow}

\newcommand{\orb}{\mathrm{orb}}

\begin{document}

\title[A Boothby--Wang theorem for Besse contact manifolds]{A Boothby--Wang theorem for Besse contact manifolds}

\author[M. Kegel]{Marc Kegel}
\address{Institut f\"ur Mathematik der Humboldt-Universit\"at zu Berlin, Unter den Linden 6, 10099 Berlin, Germany}
\email{kegemarc@math.hu-berlin.de}

\author[C. Lange]{Christian Lange}
\address{Mathematisches Institut der Universit\"at K\"oln, Weyertal 86-90, Raum -103, 50931, K\"oln, Germany}
\email{clange@math.uni-koeln.de}

\date{\today}
\subjclass[2010]{57R18, 57R17, 53D35}

\begin{abstract}
A Reeb flow on a contact manifold is called Besse if all its orbits are periodic, possibly with different periods. We characterize contact manifolds whose Reeb flows are Besse as principal $S^1$-orbibundles over integral symplectic orbifolds satisfying some cohomological condition. Apart from the cohomological condition this statement appears in the work of Boyer and Galicki in the language of Sasakian geometry \cite{Boyer_Galicki:2008}. We illustrate some non-commonly dealt with perspective on orbifolds in a proof of the above result. More precisely, we work with orbifolds as quotients of manifolds by smooth Lie group actions with finite stabilizer groups. By introducing all relevant orbifold notions in this equivariant way we avoid patching constructions with orbifold charts.

As an application, and building on work by Cristofaro-Gardiner--Mazzucchelli, we deduce a complete classification of closed Besse contact $3$-manifolds up to strict contactomorphism.
\end{abstract}
\keywords{Besse contact manifolds, Boothby--Wang theorem, symplectic orbifold, orbibundles, periodic Reeb flow}

\maketitle

\section{Introduction}\label{sec:Intro}

Recall that a contact manifold is a pair $(M,\alpha)$ of a smooth, $2n+1$-dimensional manifold $M$ and a $1$-form $\alpha$ on $M$, the so-called contact form, such that $\alpha \wedge (d\alpha)^n$ is nowhere vanishing. In the following we always assume $M$ to be connected. The Reeb vector field $R$ on $M$ is the unique vector field satisfying $\alpha(R) \equiv 1$ and $d\alpha(R,\cdot)\equiv 0$; it generates the so-called Reeb flow $\phi^{\alpha}_t : M \To M$~\cite{Geiges:2008}.

Natural examples of Reeb flows are geodesic flows on unit sphere bundles of Riemannian, or more generally Finsler manifolds~\cite{Geiges_Zehmisch_Doerner:2017}, and on Riemannian orbifolds with isolated singularities~\cite{Lange:2018b}. A contact manifold $(M,\alpha)$ is called \emph{Besse} if all its Reeb orbits are periodic. By results of Sullivan and Wadsley the Reeb flow of a Besse contact manifold is actually periodic itself~\cite{Wadsley:1975,Sullivan:1978}, see Section \ref{sub:almost_regular_Reeb}. If in addition all orbits have the same minimal period, then the contact manifold $(M,\alpha)$ is called \emph{Zoll}. Zoll Reeb flows generalize the much studied class of geodesic flows of Zoll metrics~\cite{Besse}, and are closely related to systolic geometry, see~\cite{MR3192037,ABHS2:2017,ABHS:2018, Abbondandolo_Benedetti:2019} and the references therein.

According to a classical result by Boothby and Wang Zoll contact manifolds $(M,\alpha)$ are completely understood in terms of prequantization bundles over integral symplectic manifolds~\cite{Boothby:1958}, cf.~\cite{Albers_Gutt_Hein:2018}. More precisely, a symplectic manifold $(W,\omega)$ with integral form $[\omega/2\pi]$ gives rise to principal $S^1$-bundles $M\To W$, one for each integral lift of $-[\omega/2\pi]$ to $H^2(M;\Z)$, with connection $1$-forms $\alpha$ on $M$ which are contact and whose Reeb vector fields generate the $S^1$-action on $M$. Conversely, every Zoll contact manifold arises in this way. Historically, this result was one of the first general constructions for contact manifolds and is used in several other constructions as well, see e.g.~\cite{MR3310929,Hamilton:2013}.

In the more general orbifold setting, which the Besse condition naturaly gives rise to, several special cases of a corresponding result have been investigated \cite{MR455019,Thomas:1976,Boyer:2000,Koert:2014}. A fully analogous result, as we later learned, seems to have been folklore for some time and is explicitly stated in a book by Boyer and Galicki in the language of Sasakian geometry, see Theorem~6.3.8, Theorem~7.1.3 and Theorem~7.1.6 in \cite{Boyer_Galicki:2008}. A proof of it together with several applications will also appear in the upcoming thesis of G. Placini \cite{Placini:2020}. In the case when the total space is a manifold this result applies to what we call a Besse contact manifold and adds to several further interesting characterizations~\cite{Thomas:1976,Mazzucchelli_Gardiner:2019,Mazzucchelli_Ginzburk_Basak:2019}. For instance, a contact manifold is Besse if and only if it is almost regular, which is a local condition on the flow~\cite{Thomas:1976}, see Section~\ref{sub:almost_regular_Reeb}. Moreover, $3$-dimensional, and conjecturally also higher-dimensional Besse contact manifolds admit an intriguing characterization in terms of their Reeb period spectrum~\cite{Mazzucchelli_Gardiner:2019}. Other spectral characterizations are discussed in \cite{Mazzucchelli_Ginzburk_Basak:2019}. 

In the this context we will use the orbifold Boothby--Wang theorem to prove a complete classification of closed Besse contact $3$-manifolds up to strict contactomorphism, see Theorem \ref{cor:Seifert_fibrations} and Corollary~\ref{cor:Besse_classification}. Moreover, we add a cohomological condition to the orbifold Boothby--Wang correspondence which characterizes those bundles whose total space is a manifold, and which may therefore be useful to detect new examples of Besse contact manifolds. This characterization relies on a general cohomological characterization of orbifolds among manifolds and an application of the Gysin sequence. On the technical side we explain how to think of a smooth orbifold as being represented by an almost free action (i.e.\ with finite isotropy groups) of a Lie group $G$ on a manifold $M$ and give equivariant interpretations of several orbifold notions like orbibundles. We illustrate this viewpoint in a proof of the orbifold Boothby--Wang theorem that avoids patching constructions with orbifold charts and does not refer to additional Sasakian structures. In this way the proof becomes more globally flavoured and focused on the essential objects. 

The first part of this correspondence reads as follows. Here the orbifold cohomology $H_{\orb}^*(\OO)$ of an orbifold $\OO=M//G$ is defined as the $G$-equivariant cohomology of $M$ and does not depend on the specific representation of $\OO$ in terms of $M$ and $G$, see Section~\ref{sub:cohommology}. We comment on the other relevant notions below.

\begin{thm}[cf.\ Thm.~4.3.15, Thm.~6.3.8 and Thm.~7.1.3 in \cite{Boyer_Galicki:2008}] \label{thm:main2}
Let $(M^{2n+1},\alpha)$ be a Besse contact manifold. Then, after rescaling by a suitable constant, the Reeb flow of $\alpha$ has period $2\pi$ and $\alpha$ is the connection $1$-form of a corresponding principal $S^1$-orbibundle $\pi\colon M \rightarrow \OO$ over a symplectic orbifold $(\OO,\omega)$, with $\omega$ the curvature form of $\alpha$ and $-\omega/2\pi$ representing the real Euler class $e_{\R} \in H_{\orb}^{2}(\OO;\R)$ of $\pi\colon M \rightarrow \OO$. Moreover, the integral Euler class $e$ of $\pi\colon M \rightarrow \OO$ induces isomorphisms
\begin{equation} \label{eq:cohom_condition}
e \cup \cdot \colon H_{\orb}^i(\OO;\Z) \xrightarrow{\sim}  H_{\orb}^{i+2}(\OO;\Z)
\end{equation} 
for all $i\geq 2n+1$.
\end{thm}
In the theorem $\OO$ is represented by an almost free $S^1$-action on $M$, which is induced by the Reeb vector field $R$ of $\alpha$. Let us add some more clarifying remarks.
\begin{itemize}
\item A \emph{principal $S^1$-orbibundle} whose total space $M$ is a manifold is the same as an almost free $S^1$-action on $M$ (for the general definition see Section~\ref{sec:orbifolds}). In the theorem the $S^1$-action is the one induced by the Reeb flow. In this case a \emph{connection $1$-form} $\alpha$ is a $1$-form satisfying $\alpha(R)\equiv 1$ and $d\alpha(R,\cdot)\equiv 0$. Its \emph{curvature form} is the \emph{$S^1$-basic} $2$-form $d\alpha$. Here $S^1$-basic means that the $2$-form is $S^1$-invariant and vanishes in the $S^1$-direction $R$. The basic cohomology class of $\omega$ (see Section~\ref{sub:basic_coh}) can be canonically identified with an element in $H_{\orb}^{2}(\OO;\R)$ via the equivariant de Rham theorem, see Section~\ref{sub:basic_coh}. The form $\omega$ being \emph{symplectic} means, in our language, that it is a closed basic $2$-form on $M$ satisfying $\omega^n\neq 0$. 
\item The orbifold $\OO$ in the theorem can also be seen as a \emph{symplectic reduction} (see \cite{Mardsec_Weinstein:1974,Holm:2008}) of the symplectisation $(\R_{>0} \times M,d(r\alpha))$ performed on a level set of the Hamiltonian $H(r,x)=r$. We will make this more precise in Remark \ref{rem:symplectic_reduction}.
\item Condition~(\ref{eq:cohom_condition}) actually characterizes principal $S^1$-orbibundles over any orbifold whose total space is a manifold in terms of their integral Euler class (defined in Section~\ref{sub:cohommology}), see Theorem \ref{thm:mfd_orbi_characterization}. Examples for which this condition is satisfied are discussed at the end of the introduction. 
\end{itemize}
In Theorem~\ref{thm:main2} the fact that the quotient is an orbifold to which $d \alpha$ descends as a symplectic form was observed by Weinstein in~\cite{MR455019}, generalizing a statement by Thomas that only holds over the regular part~\cite[Theorem~2]{Thomas:1976}. The first part of Theorem~\ref{thm:main2} appears in the work of Boyer and Galicki on Sasakian geometry, see~\cite{Boyer:2000} and Theorem~6.3.8 and~7.1.3 in~\cite{Boyer_Galicki:2008}. Their results are stated in terms of quasi-regular (i.e. almost regular, see Section \ref{sub:almost_regular_Reeb}) K-contact manifolds. A K-contact manifold is a metric contact manifold whose Reeb vector field is Killing, see~\cite{Boyer:2000}. For any Besse contact form $\alpha$ one can average a compatible metric via the $S^1$-action induced by the Reeb to obtain a K-contact structure. Conversely, any quasi-regular K-contact form is Besse, see \ Section~\ref{sub:almost_regular_Reeb}.

The converse construction to Theorem~\ref{thm:main2} works as follows.

\begin{thm}[cf.\ Thm.~7.1.6 in \cite{Boyer_Galicki:2008}]\label{thm:main1}
Let $(\OO^{2n},\omega)$ be a symplectic orbifold with integral symplectic form $\omega/2\pi$, i.e.
\begin{equation*}
[\omega/2\pi] \in \mathrm{Im}\big(H_{\orb}^2(\OO;\Z)\To H_{\orb}^2(\OO;\R)\big).
\end{equation*}
Then for every integral lift $e=e_{\Z}$ of $-[\omega/2\pi]$ in $H_{\orb}^2(\OO;\Z)$ that satisfies the cohomological condition~(\ref{eq:cohom_condition}) for all sufficiently large $i$, there exists a principal $S^1$-orbibundle $X \rightarrow \OO$ with Euler class $e$ and a connection $1$-form $\alpha$ on $X$ with the following properties:
\begin{itemize}
\item $X$ is a manifold
\item $\alpha$ is a Besse- and an almost regular contact form (see Section~\ref{sub:almost_regular_Reeb}),
\item the curvature form of $\alpha$ is $\omega$,
\item the vector field $R$ defining the principal $S^1$-action on $X$ coincides with the Reeb vector field of $\alpha$.
\end{itemize}
\end{thm}
We point out that the seemingly weaker cohomological condition in Theorem~\ref{thm:main1} already implies the one in Theorem~\ref{thm:main2}. A special version of this statement for certain classes of symplectic orbifolds, which are for instance K\"ahler and only have a single nontrivial isotropy group, first appeared in~\cite[Theorem~3]{Thomas:1976}. Apart from the cohomology condition the theorem is stated by Boyer and Galicki in~\cite[Thm~7.1.6]{Boyer_Galicki:2008} in the language of almost K\"ahler orbifolds and $K$-contact structures. We stress however that the class $[\omega/2\pi]$ does not uniquely determine the bundle $X$ as the formulation in~\cite{Boyer_Galicki:2008} suggests. Only its integral lift $e$ does, cf. Example~\ref{ex:orbibundle}. Moreover, a related construction of almost regular contact manifolds via fibered Dehn twists is given in~\cite[Theorem~6.5]{Koert:2014}. For an integral Euler class that does not satisfy condition (\ref{eq:cohom_condition}) the construction would yield a contact orbifold with periodic Reeb flow (cf.~\cite[Thm.~7.1.6]{Boyer_Galicki:2008}), but we will not make this notion precise here. Since the characterization in terms of~(\ref{eq:cohom_condition}) might be useful elsewhere, we explicitly state it here, the main contribution being due to Quillen, cf.~\cite[Theorem~7.7, p.~568]{Quillen:1971}.

\begin{thm}[cf.~Thm.~7.7 in \cite{Quillen:1971}]\label{thm:mfd_orbi_characterization}  An $n$-dimensional orbifold $\OO$ is a manifold if and only if $H_{\orb}^i(\OO;\Z)=0$ for all sufficiently large $i$. The total space of a principal $S^1$-orbibundle  over $\OO$  with integral Euler class $e$ is a manifold if and only if condition~(\ref{eq:cohom_condition}) is satisfied for all sufficiently large $i$. In particular, an orbifold $\OO$ can be represented by an $S^1$-action on a manifold if and only if there exists some cohomology class $e\in H_{\orb}^{2}(\OO;\Z)$ which satisfies this condition.
\end{thm} 

For example, quaternionic weighted projective spaces (see e.g.~\cite{Lange_Radeschi_Amann:2018}) have a cohomology ring generated by an element of degree $4$, and can thus not be represented by an $S^1$-action on a manifold.

A Besse contact manifold $(M,\alpha)$ induces an almost free $S^1$-action and an orientation on $M$. Conversely, one might ask which such actions, or in other words which Seifert fibrations on an odd-dimensional, orientable smooth manifold $M$ are induced by a Besse contact form $\alpha$ on $M$. For instance, the trivial fibration of $S^1\times S^2$ cannot be realized through a Besse contact form since such a form would induce an exact symplectic form on $S^2$ which is impossible by Stokes theorem. The same argument shows that the existence of a closed hypersurface meeting the fibration everywhere transversely yields an obstruction. 

Using the orbifold Boothby--Wang result we deduce the following characterization, showing that in dimension three the presence of a closed surface transversely to the fibration is virtually the only obstruction. The first part of this result can also be found in~\cite[Proposition~5.5]{Geiges:2020} in the context of geodesible vector fields. The second part about $3$-dimensional Seifert fibrations provides in combination with~\cite[Theorem~1.5]{Mazzucchelli_Gardiner:2019} a complete classification of Besse contact $3$-manifolds up to strict contactomorphism, see Corollary~\ref{cor:Besse_classification}.

\begin{thm}\label{cor:Seifert_fibrations}
A Seifert fibration $M\rightarrow B$ of a closed, orientable $2n+1$-dimensional manifold $M$ can be realized by a Reeb flow if and only if the corresponding Euler class in $H_{\orb}^2(B;\Z)$ maps to a class in $H_{\orb}^2(B;\R)$ that can be represented by a symplectic form. In the $3$-dimensional case this is equivalent to each of the following conditions
\begin{compactenum}
\item the real Euler class of the fibration $M\rightarrow B$ is nontrivial,
\item the fibration $M\rightarrow B$ is not finitely covered by a trivial fibration $S^1\times \Sigma\To \Sigma$ over an orientable surface $\Sigma$.
\end{compactenum}
\end{thm}

In particular, in the $3$-dimensional case condition $(2)$ is satisfied if $M$ has a finite fundamental group. For instance, all Seifert fibrations of lens spaces except two exceptional examples with nonorientable base~\cite{Geiges_Lange:2018} can be realized through a Reeb flow. Moreover, we remark that if $M\rightarrow B$ is a $3$-dimensional Seifert fibration with real Euler class $e_{\R}$ and Seifert invariants $(g;(a_1,b_1),\cdots,(a_n,b_n))$ (see e.g.~\cite{Jenkins_Neumann:1983,Geiges_Lange:2018,Geiges:2020}), then by~\cite[Prop.~6.1]{Geiges:2020}, cf.\ Equation~(\ref{eq:Euler_number}), the real Euler class vanishes if and only if
\[
			\sum_{i=1}^n\frac{b_i}{a_i}=0.
\]
The following statement appears in~\cite[Theorem~1.5]{Mazzucchelli_Gardiner:2019}. Here the \emph{prime period spectrum} is the collection of all minimal periods of closed Reeb orbits.
\begin{thm}[Cristofaro-Gardiner--Mazzucchelli]\label{thm:Gardiner-Mazzucchelli} Let $\alpha_1$ and $\alpha_2$ be two Besse contact forms on a closed $3$-manifold $M$. Then the prime period spectra of $\alpha_1$ and $\alpha_2$ coincide if and only if there exists a diffeomorphism $\psi:M \To M$ such that $\psi^* \alpha_2 = \alpha_1$.
\end{thm}
The prime period spectrum of a Besse contact form is up to rescaling determined by the induced Seifert fibration. Indeed, by Sullivan's and Wadsley's contributions to Theorem \ref{thm:main2}, the prime period spectrum can be read off from the period of the flow and the isotropy groups of the induced $S^1$-action. In terms of Seifert invariants $(g;(a_1,b_1),\cdots,(a_n,b_n))$ the prime period spectrum is up to rescaling given as $\{2\pi\}\cup\bigcup_{i=1}^{i=n} \{2\pi/ a_i\}\subset \R$. In particular, this result together with Theorem \ref{cor:Seifert_fibrations} provides the complete classification of Besse contact $3$-manifolds up to strict contactomorphism as claimed.
\begin{cor}\label{cor:Besse_classification} The classification of closed Besse contact $3$-manifolds up to strict contactomorphism coincides with the classification of Seifert fibrations $M\rightarrow B$ of orientable, closed $3$-manifolds satisfying condition $(1)$ or $(2)$ in Theorem \ref{cor:Seifert_fibrations}.
\end{cor}

We also remark that Theorem~\ref{cor:Seifert_fibrations} complements the results of Geiges and Gonzalo \cite{Geiges_Gonzalo:1995} on Seifert fibred $3$-manifolds which admit a contact form for which all fibres of the fibration are Legendrian, i.e.\ lie in the contact distribution, cf.~\cite[Corollary~11]{Geiges_Gonzalo:1995}. 

The cohomological condition (\ref{eq:cohom_condition}) can for instance be satisfied in case of complex weighted projective spaces $\mathbb{CP}_a^n$. These are defined as quotients of the unit sphere $\Ss^{2n+1}\subset \C^{n+1}$ by the almost free action of the unit circle $\Ss^1\subset \C^1$ of the form
\[
			z (z_0,\ldots,z_n)=(z^{a_0}z_0,\ldots, z^{a_n}z_n)
\]
for some weights $a= (a_0,\ldots,a_n) \in (\N \backslash\{0\})^n$. Indeed, their cohomology ring is computed in~\cite{Holm:2008} to be 
\[H_{\orb}^*(\mathbb{CP}_a^n;\Z)\cong \Z[u]/ \left\langle a_0\cdots a_n u^{n+1} \right\rangle,\]
where $u$ has degree two, and so any cohomology class $e=k u \in H_{\orb}^2(\mathbb{CP}_a^n;\Z)$ with $k$ and $a_0\cdots a_n$ coprime satisfies condition~(\ref{eq:cohom_condition}). In this case the total space $X$ in Theorem~\ref{thm:main1} is a lens space. Other examples can be obtained as subbundles of these. For instance, Besse Brieskorn contact manifolds occur in this way~\cite{Kwon_Koert:2016}. In fact, we are not aware of examples that do not occur in this way. In the Zoll case other examples cannot exist: For $a_0=\cdots = a_n=1$ and $e=u$ we obtain the Hopf-bundle $X=S^{2n+1} \To \mathbb{CP}^n$ and according to a statement by \'{A}lvarez Paiva and Balacheff in~\cite[Theorem~3.2.]{MR3192037}, which is based on the Boothby--Wang theorem and a result of Gromov--Tischler, every Zoll contact manifold occurs as the restriction of such a bundle to a symplectic submanifold of $\mathbb{CP}^n$ for some $n$.


\begin{qst} Can every Besse contact manifold be realized as the restriction of some bundle $X=S^{2n+1} \To \mathbb{CP}_a^n$ over a complex weighted projective space?
\end{qst}


\textbf{Acknowledgements} We would like to thank O. Goertsches and L. Zoller for discussions about equivariant cohomology and related hints to the literature. It is our pleasure to thank H. Geiges for valuable remarks. We are grateful to D. Kotschick, G. Placini and the anonymous referee for useful comments, in particular, for drawing our attention to \cite{Boyer_Galicki:2008} and pointing out a mistake in Theorem \ref{thm:main2}.

C.L. was partially supported by the DFG-grant SFB/TRR 191 “Symplectic structures in Geometry, Algebra and Dynamics”.

\section{Periodic Reeb flows}\label{sec:prelim}

In this section we provide some geometric examples for periodic Reeb flows. Moreover, we recall why almost regular and Besse Reeb flows are in fact periodic flows. Let us begin with some geometric examples.

\subsection{Geometric examples}\label{sub:Besse_geodesic} There is an infinite dimensional space of Zoll metrics on $S^2$, all of whose geodesics are closed and have the same length~\cite{Besse}, but their geodesic (Reeb) flows are all conjugated by a strict contactomorphism to the one of the standard round metric~\cite{ABHS2:2017}. This example can be generalized in two directions. On the one hand it is the starting point for Katok's construction for Zoll Finsler metrics on $S^2$~\cite{Katok:1973,Ziller:1983} whose Reeb flows~\cite{Geiges_Zehmisch_Doerner:2017} are not conjugated to the one of the round metric anymore. On the other hand there exist Riemannian orbifolds with all geodesics closed and whose unit sphere bundle is a manifold on which the geodesic flow defines a Besse Reeb flow~\cite{Lange:2018b}. For instance, the complex weighted projective spaces $\mathbb{CP}_a^n=\Ss^{2n+1}/\Ss^1$ as defined in the Section~\ref{sec:Intro} and endowed with the quotient metric have this property when all the weights are coprime~\cite{Lange_Radeschi_Amann:2018}. However, we note that by the main result of~\cite{Lange_Radeschi_Amann:2018} in the simply connected case such examples can only exist in even dimensions. Another source of examples for Besse Reeb flows are rational ellipsoids. For example, the standard $1$-form $\lambda=\frac{1}{2}\sum_{i=1,2}(x_idy_i-y_idx_i)$ on $\R^4$ restricts as a contact form to the boundary of any symplectic ellipsoid
\[
		E(a,b):=\left\{ \frac{\pi|z_1|^2}{a}+\frac{\pi|z_2|^2}{b}\leq 1 \right\} \subset \C^2 = \R^4.
\]
When $b/a$ is rational this contact form is Besse and this construction generalizes to higher dimensions.

\subsection{Besse Reeb flows}\label{sub:Besse_Reeb_flows} According to a result of Sullivan~\cite{Sullivan:1978} every Reeb flow on a contact manifold $(M,\alpha)$ with Reeb vector field $R$ is geodesible, i.e. there exists a Riemannian metric $g$ on $M$ with respect to which all Reeb orbits are unit-speed geodesics. In fact, any Riemannian metric $g$ satisfying $g(R,R)=1$ and $R\bot_g \mathrm{ker}(\alpha)$ has this property. For a simplified proof of this statement we refer the reader to~\cite[Prop.~3.3]{Geiges:2020}.
 In case of a Besse contact manifold a result of Wadsley then shows that the Reeb flow is periodic~\cite{Wadsley:1975}. Actually, this is not quite what is stated in the main theorem of~\cite{Wadsley:1975}, but it follows from the proof, see~\cite[Prop.~B.2]{Lange_Radeschi_Amann:2018}.

\subsection{Almost regular Reeb flows}\label{sub:almost_regular_Reeb}

A contact manifold $(M,\alpha)$ is called \emph{almost regular}, if there exists some positive integer $k$, and each point $x\in M$ has a cubical coordinate neighborhood $U=(z,x^1,\ldots,x^{2n})$ such that
\begin{compactenum}
\item each integral curve of the Reeb vector field $R$ passes through $U$ at most $k$ times, and
\item each component of the intersection of an integral curve with $U$ has the form $x^1=a^1,\ldots,x^{2n}=a^{2n}$, with $a^i$ constant.
\end{compactenum}
Since a closed almost regular contact manifold can be covered by finitely many of such neighborhoods, it immediately follows that its Reeb flow is Besse~\cite[Theorem~1]{Thomas:1976}. In particular, the Reeb flow of an almost regular closed contact manifold is periodic by the preceding section. This was first proven by Thomas in~\cite{Thomas:1976} modulo a small gap that was already present in the work of Boothby--Wang and fixed by Geiges in this case, see~\cite[footnote on p.~342 and Lemma~7.2.7]{Geiges:2008}. Conversely, it follows from the slice theorem that a contact manifold whose Reeb flow is induced by an almost free $S^1$-action is almost regular.

Note that it is easy to write down examples of open almost regular contact manifolds without periodic Reeb orbits. Take for example $\R^2\times[0,1]$ with standard contact form $\alpha=xdy+dz$ and identify points $(x,y,1)$ with $(x,y+1,0)$.

\section{Orbifolds}\label{sec:orbifolds}

Roughly speaking a smooth (for us always effective) orbifold can be defined as a topological Hausdorff space locally modeled by quotients of smooth actions of finite groups on smooth manifolds such that the actions satisfy certain compatibility conditions~\cite{Satake:1956,Satake:1957,Davis:2011}. Although such a description is globally not always possible, every smooth orbifold $\OO$ can be represented by an almost free action of a compact Lie group $G$ on a manifold $M$, see \cite[Cor.~1.24]{Adem:2007} or Section \ref{sub:cohommology} below. We denote such a representation as $M//G$. As a topological space $\OO$ is just the quotient space $M/G$ and all the additional data of $\OO$ is encoded in the action of $G$ on $M$. We will take this viewpoint as a definition of a smooth orbifold. Two such actions represent diffeomorphic orbifolds if and only if there exist invariant Riemannian metrics with respect to which the quotient spaces endowed with the quotient metric, which measures the distance between orbits, are isometric. Here we can take this as a definition, but it coincides with the usual notion~\cite{Lange:2018}. The dimension of an orbifold represented by an effective action of $G$ on $M$ is the difference of the dimensions of $M$ and $G$. If $G$ can be chosen to be discrete, the orbifold is called developable. To prove the independence of several notions of the specific representation of an orbifold, we need the following pullback construction. 

\begin{lem}\label{lem:pullback} Suppose a smooth orbifold $\OO$ is represented as $M_1//G_1$ and as $M_2//G_2$. Then there exists a manifold $X$ with an almost free action of $G_1 \times G_2$ such that the following equivariant diagram commutes
\[
	\begin{xy}
		\xymatrix
		{
		  G_1\times G_2 \curvearrowright X \ar[d] \ar[r] &  G_2 \curvearrowright M_2 \ar[d]  \\
		  G_1 \curvearrowright M_1 \ar[r] &\OO  
		}
	\end{xy}
\]
where the upper and the left arrow are the quotient maps for the action of $G_1$ and $G_2$, respectively.	In particular, the actions of $G_1$ and $G_2$ on $X$ are free.
\end{lem}
We emphasize that this construction differes from the usual pullback construction. Indeed, the actions of $G_1$ and $G_2$ on $M_1$ and $M_2$, respectively, are in general not free, but the lifted actions of $G_1$ and of $G_2$ on $X$ are so.
The proof of Lemma \ref{lem:pullback} is given in the appendix, Section~\ref{sec:appendix}.

\subsection{Orbibundles and orbifold cohomology}\label{sub:cohommology} The notion of a \emph{principal $S^1$-orbibundle} is defined in~\cite{Satake:1957} (see~\cite{Adem:2007} for the definition in terms of groupoids). Every principal $S^1$-orbibundle over an orbifold $M//G$ lifts to a $G$ principal $S^1$-bundle over $M$~\cite[Example~2.29]{Adem:2007}. By this we mean a principal $S^1$-bundle $P$ over $M$ to which the action of $G$ extends in such a way that it commutes with the $S^1$-action. In this case the total space of the $S^1$-orbibundle over $M//G$ is $P//G$. Conversely, a $G$ principal $S^1$-bundle over $M$ gives rise to a principal $S^1$-orbibundle over $M//G$ in the sense of~\cite{Satake:1957}. When $M_1//G_1$ and $M_2//G_2$ represent the same orbifold, then two such bundles are equivalent if and only if in the diagram of Lemma \ref{lem:pullback}, the pulled back bundles to $X$ are $(G_1\times G_2)$-equivariantly equivalent. Here we take the latter viewpoint of this equivalence as a definition. Moreover, it follows from Lemma \ref{lem:pullback} that a principal $S^1$-orbibundle whose total space $M$ is a manifold is the same as an almost free $S^1$-action on $M$. 

Recall that principal $S^1$-bundles over a manifold $B$ are classified by elements in $H^2(B;\Z)$~\cite[Theorem~13.1]{MR1249482}. An analogous correspondence holds for $S^1$-orbibundles which we now want to describe. Using Lemma \ref{lem:pullback} one can show that for an orbifold $\OO=M//G$ the homotopy type of the Borel construction $M\times_{G} EG:=(M\times EG)/G$ is independent of the specific representation of $\OO$ in terms of $M$ and $G$, and construct canonical isomorphisms between the corresponding cohomology rings, see the appendix, Section~\ref{sec:appendix}. Here $EG$ is a contractible CW-complex on which $G$ acts freely and the action of $G$ on $M\times EG$ is the diagonal action. The space $B\OO:=M\times_{G} EG$ is called a \emph{classifying space} for $\OO$ and the orbifold cohomology of $\OO$ with respect to some coefficient ring $R$ is defined as $H_{\orb}^*(\OO;R):=H_G^*(M;R):=H^*(B\OO;R)$. For the sake of concreteness we can always assume $G$ to be an orthogonal group $\mathrm{O}(m)$ (by choosing $M$ to be the orthonormal frame bundle of $\OO$ with respect to some Riemannian metric, see~\cite[Cor.~1.24]{Adem:2007}). In this case $BG$ can be taken to be a direct limit of Grassmannians $\mathrm{Gr}_m(\R^k)\subset \mathrm{Gr}_m(\R^{k+1})\subset \ldots$ with the final topology \cite{Milnor_Stasheff:1974}.

The \emph{equivariant Euler class} $e^G(P)\in H_{G}^2(M;\Z)$ of a $G$ principal $S^1$-bundle $P$ over $M$ is defined as the Euler class of the principal $S^1$-bundle over $B\OO$ obtained by pushing down the pulled back $G$ principal $S^1$-bundle $p^*P$ over $M\times EG$ to $B \OO$. Two equivalent bundle over equivalent representations of an orbifold give rise to Euler classes that are identified via canonical isomorphisms mentioned above, see Section~\ref{sec:appendix}, i.e.\ we can also view $e^G(P)$ as an orbi-Euler class $e_{\orb}\in H_{\mathrm{orb}}^2(\OO,\Z)$ associated with a principal $S^1$-orbibundle over $\OO$. It is shown in~\cite{Hattori:1976,Ignasi:2001} that $G$ principal $S^1$-bundles $P$ over $M$ are classified via their equivariant Euler class by elements in $H_{\mathrm{orb}}^2(\OO;\Z)$.  In fact, their result, which is formulated for complex line bundles, is more general covering all smooth Lie group actions. In terms of orbifolds it can be phrased as follows.

\begin{thm}[Hattori, Yoshida]\label{thm:orbibundle_class} Principal $S^1$-orbibundles over an orbifold $\OO$ are classified by elements in $H_{\mathrm{orb}}^2(\OO;\Z)$ via their Euler class.
\end{thm}

\begin{exm}\label{ex:orbibundle} Let $\OO$ be the quotient of $\C$ by the action of a cyclic group $\Z_k < \mathrm{U}(1)$. Since $\C$ is contractible, the integral cohomology ring of $\OO$ can be computed to be
\[
		H_{\orb}^*(\C/\Z_k;\Z)=H^*(\Z_k;\Z)\cong \Z[u]/\left\langle k u\right\rangle,
\]
where $u$ has degree $2$~\cite{Agol:2013,Holm:2008}. Hence, there are $k$ isomorphism classes of principal $S^1$-orbibundles over $\C/\Z_k$. The bundle with Euler class $e=lu$ can be constructed as a quotient of $\C\times S^1$ by the $\Z_k$-action $\xi(z,\lambda)=(\xi z, \xi^l \lambda)$. We see that the total space of this bundle is a manifold if and only if $l$ and $k$ are coprime in accordance with condition~(\ref{eq:cohom_condition}). In particular, for $k$ prime the only bundle whose total space is not a manifold is the trivial bundle $S^1\times \OO$.
\end{exm}

\subsection{Basic cohomology}\label{sub:basic_coh}

Let an orbifold $\OO$ be represented by an almost free action of $G$ on $M$. A differential form $\tau$ on $M$ is called ($G$-)basic if it is $G$-invariant and vanishes when contracted with vertical vector fields, i.e.\ vector fields in the vertical distribution of the projection $M\To M//G$. By the Cartan formula the differential of a basic form is again basic. The cohomology of the subcomplex of basic differential forms is called the \emph{basic $G$-cohomology} of $M$ and denoted as $\HG^*(M)$, see e.g.~\cite{GoZo:2019}. If the action of $G$ on $M$ is free, then $\OO=M//G$ is a manifold and $\HG^*(M)$ is canonically isomorphic to the de Rham cohomology $H^*_{\mathrm{dR}}(M)$~\cite[Prop.~2.5]{GoZo:2019}. If the action is only almost free we still have the equivariant de~Rham theorem saying that $H_{\orb}^*(\OO;\R)$ is canonically isomorphic to $\HG^*(M)$~\cite[Thm.~2.5.1]{Guillemin:1999}. We mention that the latter is in turn isomorphic to the de~Rham cohomology $H^*_{\mathrm{dR}}(\OO)$ of $\OO$ as defined in e.g.~\cite{Satake:1956}. More precisely, the equivariant de~Rham isomorphism looks as follows. First, the usual de~Rham theorem (applied to finite-dimensional approximations) yields an isomorphism $H_{\orb}^*(\OO;\R) \cong H^*_{\mathrm{dR}}(B\OO)$. The de~Rham cohomology of $B\OO$ is in turn canonically isomorphic to the basic $G$-cohomology $\HG^*(M\times EG)$ by~\cite[Prop.~2.5]{GoZo:2019} (again applied to finite-dimensional approximations), where the isomorphism is induced by the projection $M\times EG \To M\times_G EG$. Finally, the isomorphism between  $\HG^*(M)$ and $\HG^*(M\times EG)$ is induced by the $G$-equivariant projection $p_1 \colon M\times EG \To M$. For the convenience of the reader we sketch an argument why the latter map is indeed an isomorphism.  

\begin{lem} \label{lem:pullback_is_iso} The map $p_1^*\colon\HG^*(M) \To \HG^*(M\times EG)$ is an isomorphism.
\end{lem}

\begin{proof} Since the actions of $G$ on $M$ and $M\times EG$ are almost free, we have natural isomorphisms $\HG^*(M)\cong H^*(C_G(M))$ and $\HG^*(M\times EG)\cong H^*(C_G(M\times EG))$ induced by the map $\omega \mapsto 1\otimes \omega $ between cochain complexes. Here $C_G(M)=(S(\frak{g^*})\otimes \Omega (G))^G$ denotes the cochain complex of the \emph{Cartan model} for the equivariant cohomology (see~\cite{GoZo:2019}, Section 4 and Theorem 5.2 for the details). Hence, it suffices to prove the claim for the map $p_1^*:H^*(C_G(M))\rightarrow H^*(C_G(M\times EG))$ induced by the map $1\otimes p_1^*$ on the cochain level. The latter map respects a filtration which gives rise to a spectral sequence that computes the equivariant cohomology, see~\cite[Section~A.3]{GoZo:2019}). By the comparison theorem~\cite[Thm.~A.22]{GoZo:2019} it suffices to show that $1\otimes p_1^*$ induces an isomorphism on the first pages of the spectral sequences which are $S(\frak{g^*})\otimes H^*(M)$ and $S(\frak{g^*})\otimes H^*(M\times EG)$~\cite[Thm.~A.8]{GoZo:2019}, respectively. Going through the proof of~\cite[Thm.~A.8]{GoZo:2019} shows that this induced map is just the map $1\otimes p_1^*$. Hence, the claim follows.
\end{proof} 

We call elements in $H_{\mathrm{orb}}^*(\OO;\R)$ and $\HG^*(M)$ which lie in the image of $H_{\mathrm{orb}}^*(\OO;\Z)$ \emph{integral}. In particular, the Euler class of a principal $S^1$-orbibundle over $\OO$ gives rise to an integral class in $\HG^*(M)$, which we call the \emph{real Euler class} of the bundle. In the following subsection we describe this class in terms of differential forms on $M$.

\subsection{Real Euler class}\label{sub:real_euler} 

Recall that a connection $1$-form of a principal $S^1$-bundle $P\To M$ is a $1$-form $\alpha$ on $P$ such that $\mathcal{L}_R \alpha=0$ and $\alpha(R)=1$ holds, where $R$ is the vector field on $P$ generated by the $S^1$-action and $\mathcal{L}$ is the Lie derivative. A \emph{connection $1$-form} of a $G$ principal $S^1$-bundle $P\To M$ is a connection $1$-form of the underlying principal $S^1$-bundle which is in addition $G$-basic. It is not difficult to show that every $G$ principal $S^1$-bundle $P\To M$ admits a connection $1$-form. We will, however, only be concerned with the case in which the total space $P//G$ is a manifold and in this case one can more easily construct a connection $1$-form for the almost free $S^1$-action on $P//G$, and then pull it back to a connection $1$-form of the $G$ principal $S^1$-bundle $P\To M$.


As for principal $S^1$-bundles the differential $d\alpha$ of a connection $1$-form of a $G$ principal $S^1$-bundle $P\To M$ descends to a so-called \emph{curvature form} $\omega$ on $M$ (cf.~\cite[p.~340]{Geiges:2008}) which, in the $G$-equivariant case, is in addition $G$-basic. In particular, we have an induced cohomology class $-[\omega/2\pi]\in \HG^*(M)$. In the equivariant case we can also think of $\omega$ as an $(S^1\times G)$-basic $2$-form on $P$, or as an $S^1$-basic $2$-form on $P//G$ if it is a manifold.

\begin{lem} \label{lem:curvature_form} The curvature form $\omega$ of a $G$ principal $S^1$-bundle $P\To M$ determines an integral form $-[\omega/2\pi]\in \HG^*(M)$ which coincides with the real Euler class $e_{\R}^G(P)$ of the bundle.
\end{lem}
\begin{proof} By Lemma~\ref{lem:pullback_is_iso} and the definition of the Euler class it is sufficient to show the claim for the pulled back $G$ principal $S^1$-bundle over $M \times EG$. A connection $1$-form and the corresponding curvature form descend to the respective forms of the quotient principal $S^1$-bundle over $M \times_G EG$. Since the projection $M \times EG \To M \times_G EG$ induces an isomorphism from the de Rham cohomology of $M \times_G EG$ to the basic cohomology of $M \times EG$, the lemma follows from its non-equivariant version applied to the $S^1$-bundle over $M \times_G EG$~\cite[Theorem~13.1]{MR1249482}.
\end{proof}

%

\section{Proof of the main results}
\label{sec:proof_main_results}

\begin{prop}\label{cor:orbi_mani_charact} An $n$-dimensional orbifold $\OO$ is a manifold if and only if $H_{\orb}^i(\OO;\Z)=0$ for all $i$ sufficiently large.
\end{prop}
\begin{proof} It follows from a result of Quillen~\cite[Theorem 7.7, p.~568]{Quillen:1971} that $\OO$ is a manifold if and only if $H^*_{\orb}(\OO;\Z_p)$ is finite dimensional for all primes $p$, see~\cite[Remark~3.4]{Lange_Radeschi_Amann:2018}. In the compact, orientable case a more elementary argument for this statement is provided in  \cite[Section~3]{Lange_Radeschi_Amann:2018}. Now we observe that for any coefficients all cohomology groups are finitely generated since $B\OO$ can be approximated by finite-dimensional manifolds, see Section~\ref{sub:cohommology}. Therefore the claim now follows from the cohomological universal coefficient theorem.
\end{proof}

\begin{proof}[Proof of Theorem~\ref{thm:mfd_orbi_characterization}] The first part of the theorem is just the statement of Proposition~\ref{cor:orbi_mani_charact}. For the second part we represent $\OO$ as $M//G$ and the principal $S^1$-orbibundle as a $G$ principal $S^1$-bundle over $M$. The second claim then follows from the Gysin sequence \cite[Section~2.2]{Kochman:1996} applied to the corresponding principal $S^1$-bundle over $(M\times ES^1)/S^1$, i.e.
\[
		\cdots \rightarrow H^i(M;\Z)\rightarrow H_{\orb}^{i-1}(\OO;\Z) \xrightarrow{e\cup \cdot} H_{\orb}^{i+1}(\OO;\Z) \rightarrow  H^{i+1}(M;\Z) \rightarrow \cdots,
\]
since $H^{i}(M;\Z)=0$ for all $i>2n+1$.
\end{proof}

Recall that for an orbifold $\OO$ represented by an almost free action of $G$ on $M$ a \emph{symplectic form} on a $\OO$ is a closed basic $2$-form $\omega$ on $M$ which satisfies $\omega^n\neq 0$. In this case the pair $(M//G,\omega)$ is called a symplectic orbifold.

\begin{proof}[Proof of Theorem \ref{thm:main2}]
According to Sections~\ref{sub:Besse_Reeb_flows} and \ref{sub:almost_regular_Reeb} we can assume that the Reeb flow is periodic, and, after rescaling, has period $2\pi$. In other words, the $\R$-action defined by the flow factors through an $S^1$-action. Since the Reeb vector field has no zeros and $S^1$ has only finite proper subgroups, this action is almost free. Since $d \alpha$ is $S^1$-basic and $d\alpha^n\neq 0$, $\alpha$ is a connection $1$-form and the quotient $(M//S^1,d\alpha)$ is a symplectic orbifold.

Let $X$ be the $S^1\times S^1$-space obtained by applying Lemma~\ref{lem:pullback} to two copies $M_1$ and $M_2$ of $M$ with the same given $S^1$-action, and write $\pi:X\To M$ for the projection map. In order to prove that $[d\alpha] \in H_{\mathrm{bas}\:S^1}(M) \cong H^2_{\mathrm{orb}}(\OO)$ represents the real Euler class of the principal $S^1$-orbibundle $M\To M//S^1$, we only need to show, given Lemma~\ref{lem:curvature_form}, that the images of $[d\alpha]$ and of $[d\pi^*\alpha]$ in $H^2((M\times S^1)/S^1;\R)$ and $H^2((X\times ES^1\times ES^1)/S^1;\R)$, respectively, are identified by the canonical ismorphism between $H^2((M\times S^1)/S^1;\R)$ and $H^2((X\times ES^1\times ES^1)/S^1;\R)$ that is induced by the projection $\pi$, see the appendix, Section~\ref{sec:appendix}. Unrolling the definitions shows that this is implied by Stokes theorem. The last claim is an application of Theorem~\ref{thm:mfd_orbi_characterization}.
\end{proof}

\begin{rem}[Symplectic reduction viewpoint] \label{rem:symplectic_reduction}
In the introduction we claimed that the orbifold $\OO$ in Theorem~\ref{thm:mfd_orbi_characterization} can also be seen as a symplectic reduction of the symplectisation $(\R_{>0} \times M,d(r\alpha))$ performed on a level set of the Hamiltonian $H(r,x)=r$. Indeed, extending the $S^1$-action on $M$ trivially to the $\R_{>0}$ factor yields an action with action vector field $(0,R)$ that leaves the Hamiltonian $H$ invariant. Because of $\iota_{(0,R)} d(r\alpha) = dr=dH$, this action is Hamiltonian and $H$ specifies its moment map $\mu:\R_{>0} \times M \To \R^*$ via $H=\left\langle\mu, 1\right\rangle$. Note that any $r\in \R_{>0}$ is a regular value of $H$. Hence, for any $r\in \R_{>0}$ the orbifold $\OO$ arises as a symplectic reduction $\mu^{-1}(r)//S^1$ as claimed, cf. \cite[Thm.~1.11]{Holm:2008},\cite{Mardsec_Weinstein:1974}.
\end{rem}

\begin{proof}[Proof of Theorem~\ref{thm:main1}]
Let $(\OO=M//G,\omega)$ be a symplectic orbifold of dimension $2n$ with integral symplectic form $\omega/2\pi$. For every integral cohomology class $e\in H_{\orb}^2(\OO;\Z)$ mapping to $-[\omega/2\pi] \in \HG^2(M)$ and satisfying condition~(\ref{eq:cohom_condition}) for sufficiently large $i$ we get by Theorem~\ref{thm:orbibundle_class} a unique principal $S^1$-orbibundle $\pi: X \rightarrow \OO$ with integral Euler class $e$ and real Euler class $-[\omega/2\pi]$. By condition~(\ref{eq:cohom_condition}) and Theorem~\ref{thm:mfd_orbi_characterization} the space $X$ is a manifold, and so we can think of the bundle as an almost free $S^1$-action on $X$ with $X//S^1= \OO$. In particular, we can assume that $M=X$ and $G=S^1$. Let $\alpha'$ be a connection $1$-form for this bundle. As in the proof of Theorem~\ref{thm:main2} it follows with Lemma~\ref{lem:curvature_form} that the corresponding curvature form $\omega'$ represents the real Euler class of the bundle. Hence, there is a basic $1$-form $\beta$ on $X$ such that $\omega-\omega'=d\beta$. Set $\alpha=\alpha'+\pi^* \beta$. This is also a basic connection $1$-form, and it satisfies $d\alpha= \pi^* \beta$. Since $\omega$ is a symplectic form on $\OO$ it follows that $\alpha$ is a contact form on $X$.

\end{proof}

Let us make some remarks before we prove Theorem~\ref{cor:Seifert_fibrations}. For an almost free $S^1$-action on an orientable $3$-manifold $M$ with base $B=M//S^1$ and connection $1$-form $\alpha$ integration induces a homomorphism:
\[
		\left\langle \cdot , [B] \right\rangle: H^2_{\mathrm{bas}\: S^1}(M) \To \R, [\omega] \mapsto \left\langle [w] , [B] \right\rangle:= \frac{1}{2\pi}\int \alpha\wedge \omega.
\]
We mention that under the canonical identification of $H^2_{\mathrm{bas}\: S^1}(M)$ with the orbifold de Rham cohomology of $B=M//S^1$ this map amounts to integration over the base orbifold $M//S^1$ as defined in~\cite{Satake:1956}. In particular, this homomorphism is in fact an isomorphism~\cite[Thm.~3]{Satake:1956}. In the present special case this follows from the fact that $H^2_{\mathrm{bas}\: S^1}(M)$ is isomorphic to $H^2_{S^1}(M;\R)\cong H^2(M/S^1;\R)\cong \R$ and that one can easily construct a nowhere vanishing basic $2$-form $\omega$ on $M$, because of the orientability assumptions, for which $\alpha \wedge \omega$ is then a volume form of $M$. Such an $\omega$ is a symplectic form on $M//S^1$. In particular, we see that a class $[\omega]$ in $H^2_{\mathrm{bas}\: S^1}(M)$ can be represented by such a symplectic form if and only if $\left\langle [\omega] , [B] \right\rangle \neq 0$, and hence if and only if $[\omega]\neq 0$ in $H^2_{\mathrm{bas}\: S^1}(M)$. Moreover, we record that this property is invariant under finite coverings, i.e. if $\hat{M}\To M$ is a finite covering, then the $S^1$-action on $M$ is covered by an $S^1$-action on $\hat M$, and a class in $H^2_{\mathrm{bas}\: S^1}(M)$ has a symplectic representative if and only if its pull back in $H^2_{\mathrm{bas}\: S^1}(\hat M)$ has so.

Finally, we also mention that if the fibration $M\rightarrow B$ has real Euler class $e_{\R}$ and Seifert invariants $(g;(a_1,b_1),\cdots,(a_n,b_n))$ (see e.g.~\cite{Jenkins_Neumann:1983,Geiges_Lange:2018,Geiges:2020}), then by~\cite[Prop.~6.1]{Geiges:2020} we have that
\begin{equation}\label{eq:Euler_number}
			\left\langle e_{\R} , [B] \right\rangle = - \sum_{i=1}^n\frac{b_i}{a_i}.
\end{equation}

\begin{proof}[Proof of Theorem~\ref{cor:Seifert_fibrations}]
The first part of the theorem is a direct consequence of Theorem~\ref{thm:main2} and Theorem~\ref{thm:main1}. For the second part we first note that if $M$ is $3$-dimensional, then $B$ is a closed $2$-orbifold equipped with a natural orientation. There exists a finite orbifold covering $\hat B$ of $B$ with torsion free $H_{\orb}^2(\hat B;\Z)$. 
Indeed, by the universal coefficient theorem we have
\[
H_{\orb}^2(\hat B;\Z)_{\mathrm{tor}}=H^{\orb}_1(\hat B;\Z)_{\mathrm{tor}}=(\pi_1^{\orb}(\hat B)_{\mathrm{ab}})_{\mathrm{tor}},
\]
and every closed $2$-orbifold is either finitely covered by a simply connected $2$-orbifold or by a surface~\cite[Thm.~2.5]{Scott:1983}. As observed above, a class in $H^2_{\orb}(B;\Z)$ can be represented by a symplectic form if and only if it pulls back to a non-trivial class in $H^2_{\orb}(\hat B;\R)\cong H^2_{\orb}(\hat B;\Z)$. Since this image is the integral Euler class of the pulled back almost free $S^1$-action on $\hat M$ with quotient $\hat B$, this happens if and only of this action is not trivial.
Hence, if a given Seifert fibration on $M$ with orientable base is not covered by a trivial fibration, then it can be realized by a Reeb flow. Conversely, assume that such a fibration can be realized by a Reeb flow and is finitely covered by a trivial fibration $\hat M\cong S^1\times \Sigma\rightarrow \Sigma$. The base of this fibration $\hat B =\Sigma$ is a surface covering $B$. Moreover, the real Euler class of the corresponding $S^1$-bundle vanishes in contradiction to the fact that its preimage in $H^2_{\orb}(B;\R)$ is nontrivial.
\end{proof}

%

\section{Appendix}\label{sec:appendix}

In this section we prove Lemma~\ref{lem:pullback} and confirm the independence of the orbifold cohomology and the orbifold Euler class from specific representations.

\begin{proof}[Proof of Lemma~\ref{lem:pullback}] Suppose we have two representations of $\OO^n$ in terms of two actions $G_1\curvearrowright M_1$ and $G_2 \curvearrowright M_2$. Then there are invariant Riemannian metrics with respect to which the corresponding quotient spaces $M_1/G_1$ and $M_2/G_2$ are isometric. Let $\mathrm{Fr}^h(M_i)$ be the principal $\mathrm{O}(n)$-bundle over $M_i$ consisting of orthonormal $n$-frames in the horizontal distribution of the projection $M_i \To \OO$, $i=1,2$. The natural $\mathrm{O}(n)$-action on these spaces commutes with the free actions of $G_1$ and $G_2$, respectively. The quotient spaces $\mathrm{Fr}^h(M_1)/G_1$ and $\mathrm{Fr}^h(M_2)/G_2$ are naturally identified with the orthonormal frame bundle $\mathrm{Fr}(\OO)$ of $\OO$. We consider the space
\[
		\bar{X}=\{(x,y)\in \mathrm{Fr}^h(M_1) \times \mathrm{Fr}^h(M_2)\mid G_1x=G_2y \in \mathrm{Fr}(\OO) \}
\]
with the induced actions of $G_1$, $G_2$ and the diagonal action of $\mathrm{O}(n)$. Then the space $X=\bar{X}/\mathrm{O}(n)$ with the induced action of $G_1\times G_2$ satisfies all conditions in the lemma.
\end{proof}

Now let us look at the independence of the Borel construction of the specific representation $M_i//G_i$, $i=1,2$ of $\OO$. In view of Lemma~\ref{lem:pullback} it suffices to compare the Borel constructions of $M_1//G_1$ and of $X//G$, where $G=G_1 \times G_2$. In this case the natural projection from the Borel construction 
\[(X\times EG_1\times EG_2)/(G_1 \times G_2)\]
to $(X \times EG_1)/(G_1 \times G_2)=(M_1 \times EG_1)/G_1$ defines a fibre bundle with contractible fibre $EG_2$. In particular, it induces a homotopy equivalence between these spaces and an isomorphisms in cohomology. Here we have taken the independence of a specific classifying space $EG$ for granted; if $EG$ and $\tilde{EG}$ are two different models of this classifying space, then so is $EG \times \tilde{EG}$ and the same argument as above shows the independence of $EG$.

It remains to observe that the canonical isomorphism 
\[\pi^*: H^2((M_1 \times EG_1)/G_1;\Z) \To H^2((X\times EG_1\times EG_2)/(G_1 \times G_2);\Z)\]
induced by the projection maps the integral Euler class of a $G_1$ principal $S^1$-bundle over $M_1$ to the integral Euler class of the pulled back $G$ principal $S^1$-bundle over $X$. Indeed, these are the Euler classes of a principal $S^1$-bundle over $(M_1 \times EG_1)/G_1$ and its pulled back bundle over $(X\times EG_1\times EG_2)/(G_1 \times G_2)$. This shows the independence of the Euler class of the specific representation of $\OO$ as claimed in Section~\ref{sub:cohommology}.

\bibliography{_biblio}
\bibliographystyle{plain}

\end{document}